\documentclass[preprint,12pt]{elsarticle}
\usepackage{amssymb}
\usepackage{amsthm}
\journal{Elsevier}
\usepackage[dvipsnames]{xcolor}

\usepackage{amsmath}
\usepackage{amsfonts}
\usepackage{amssymb}
\usepackage{xcolor}
\usepackage{enumerate}
\usepackage{lineno}
\usepackage{tikz}

\newtheorem{definition}{Definition}[section]
\newtheorem{proposition}[definition]{Proposition}
\newtheorem{lemma}[definition]{Lemma}
\newtheorem{theorem}[definition]{Theorem}
\newtheorem{corollary}[definition]{Corollary}

\newdefinition{example}[definition]{Example}
\newdefinition{remark}[definition]{ Remark}
\newdefinition{problem}[definition]{ Problem}
\newdefinition{question}[definition]{Question}
\newdefinition{fact}[definition]{Fact}
\newproof{pot}{Proof}

\begin{document}

\begin{frontmatter}
\title{Hyperspaces $C(p,X)$ of finite graphs}

\author{Florencio~Corona-V\'azquez }
\ead{florencio.corona@unach.mx}

\author{Russell~Aar\'on~Qui\~nones-Estrella}
\ead{rusell.quinones@unach.mx}

\author{Javier~S\'anchez-Mart\'inez \corref{cor1}}
\ead{jsanchezm@unach.mx}

\author{Hugo~Villanueva}
\ead{hugo.villanueva@unach.mx}

\cortext[cor1]{Corresponding author}

\address{Universidad Aut\'onoma de Chiapas, Facultad de Ciencias en F\'isica y Matem\'aticas, Carretera Emiliano Zapata Km. 8, Rancho San Francisco, Ter\'an, C.P. 29050, Tuxtla Guti\'errez, Chiapas, M\'exico.}

\begin{abstract}
Given a continuum $X$ and $p\in X$, we will consider the hyperspace $C(p,X)$ of all subcontinua of $X$ containing $p$ and the family $K(X)$ of all hyperspaces $C(q,X)$, where $q\in X$. In this paper we give some conditions on the points $p,q\in X$ to guarantee that $C(p,X)$ and $C(q,X)$ are homeomorphic, for finite graphs $X$. Also, we study the relationship between the homogeneity degree of a finite graph $X$ and the number of topologically distinct spaces in $K(X)$, called the size of $K(X)$. In addition, we construct for each positive integer $n$, a finite graph $X_n$ such that $K(X_n)$ has size $n$, and we present a theorem that allows to construct finite graphs $X$ with a degree of homogeneity different from the size of the family $K(X)$.
\end{abstract}

\begin{keyword}
Continua \sep  finite graphs \sep hyperspaces.
\MSC Primary \sep 54B20 
\sep 54B05 
\sep 54F65 
\end{keyword}

\end{frontmatter}
\section{Introduction}
A \textit{continuum} is a nonempty compact connected metric space. Given a continuum $X$, by a \emph{hyperspace} of $X$ we mean a specified collection of subsets of $X$ endowed with the Hausdorff metric (see Theorem \ref{metrica}).   In the literature, some of the most studied hyperspaces are the following:   
\begin{align*}
2^{X} & = \{ A \subseteq X : A \text{ is nonempty and closed}\},        \\
C(X) & = \{ A \subseteq X : A \text{ is nonempty, connected and closed}\}.
\end{align*}

$2^{X}$ is called the \textit{hyperspace of closed subsets} of $X$ whereas that $C(X)$ is called the \textit{hyperspace of subcontinua} of $X$. These have been amply studied to characterize topological properties of $X$ through them and vice versa. The readers are referred to \cite{Illanes(1999)} for more information about this topic.

For the purpose of this paper, given a continuum $X$, $A\in C(X)$ and a point $p$ of $X$, we focus our attention also to the following hyperspaces:
\begin{align*}
C(A,X) & = \{B\in C(X):A\subset B\} , \\
C(p,X) & = \{ A \in C(X) :p\in A \}, \\
K(X) & = \{C(x,X):x\in X\}.
\end{align*}

The topological structure of the hyperspaces $C(A,X)$ and $C(p,X)$ have been recently studied, for example, in \cite{Eberhart(1978)}, \cite{Martinez(2007)}, \cite{Pellicer(2003)}, \cite{Pellicer(2005)}, and \cite{Pellicer(2005b)}. In \cite{Pellicer(2003)} P. Pellicer gives characterizations of the class of continua $X$ for which $K(X)$ coincides with $K(I)$ or $K(S)$, where $I$ denotes the unit interval and $S$ is a simple closed curve. As a useful tool to characterize hyperspaces of the form $K(X)$, is distinguishing the structure of the hyperspaces $C(p,X)$ where $p\in X$, see for example \cite{Pellicer(2007)},   in this way, we consider in $K(X)$ the following natural equivalence relation: $C(p,X)\sim C(q,X)$ if and only if $C(p,X)$ is homeomorphic to $C(q,X)$. Given a positive integer $n$ and a continuum $X$, we say that \textit{$K(X)$ has size $n$} if the quotient $K(X)/\sim$ has cardinality $n$. In this paper we study this relation in the class of finite graphs as well as the size of $K(X)$. We show that $K(X)$ having size $n$ is not equivalent to $X$ being $\frac{1}{n}$-homogeneous. We present, for each positive integer $n$, a finite graph $X_n$ such that $K(X_n)$ has size $n$, we also present in Theorem \ref{pseudosymmetric} a way to construct continua $X$ with a degree of homogeneity different from the size of $K(X)$. The homogeneity degree of continua has been studied extensively, see for example \cite{LopezPellicerSantiago(2010)}, \cite{NeumannPellicerPuga(2006)}, \cite{Pellicer(2008)} and \cite{SantiagoTapia(2018)}, being a useful property to characterize classes of continua. Regarding \cite[Corollary 3.12, p. 1006]{Pellicer(2007)}, in Theorem \ref{maintheorem}, we present a similar result in the class of finite graphs, both results aims to find cells in $C(p,X)$.

\section{Preliminaries}

Let $X$ be a continuum with metric $d$. Given $\varepsilon>0$ and $p\in X$ we denote as customary $B_{\varepsilon}(p)=\{x\in X:d(x,p)<\varepsilon\}$ and if $A\subset X$ 
\begin{center}
$N(\varepsilon,A)=\left\{x\in X:\textrm{there exists $y\in A$ such that }x\in B_{\varepsilon}(y)\right\}$.
\end{center} 
If $A$ and $B$ are two closed subsets of $X$, remember that the \textit{Hausdorff distance} between $A$ and $B$ is given by:
\begin{center}
$H(A,B)=\inf\{\varepsilon>0: A\subset N(\varepsilon,B)$ and $ B\subset N(\varepsilon,A)\}$.
\end{center} 

\begin{theorem}\label{metrica}\cite[Theorem 2.2, p. 11]{Illanes(1999)} 
If $(X,d)$ is a metric compact space, then $H$ is a metric for $2^{X}$.
\end{theorem}

Since $2^{X}$ equipped with the metric $H$ is a continuum (cf. \cite[Corollary 14.10, p.114]{Illanes(1999)}), we also consider $H_{2}$ the Hausdorff metric by $2^{2^{X}}$ induced by $H$. The hyperspaces $C(X)$, $C(p,X)$ and $C(A,X)$ are considered as subspaces of $2^{X}$ and $K(X)$ as a subspace of $2^{2^{X}}$. The topological structure of the hyperspace $K(X)$ was studied first by P. Pellicer in \cite{Pellicer(2005)}, the author shows that the hyperspace $K(X)$ is not always a continuum and gives conditions to ensure that $K(X)$ is compact, connected, arcwise connected and locally connected.

By a \textit{finite graph} we mean a continuum $X$ which can be written as the union of finitely many arcs, any two of which are either disjoint or intersect only in one or both of their end points. Given a positive integer $n$,  a \textit{simple $n$-od} is a finite graph, denoted by $T_n$, which is the union of $n$ arcs emanating from a single point, $v$, and  otherwise disjoint from each another. The point $v$ is called the \textit{vertex} of the simple $n$-od. A simple $3-od$, $T_{3}$, will be called a \textit{simple triod}. A \textit{tree} is a finite graph without simple closed curves. A $n$\emph{-cell} is any space homeomorphic to $[0,1]^n$.

Given a finite graph $X$, $p\in X$ and a positive integer $n$, we say that $p$ \textit{is of order $n$ in $X$}, denoted by 
$ord(p,X)=n$, if $p$ has a closed neighborhood which is homeomorphic to a simple $n$-od having $p$ as the vertex. 
The points of order $1$, $2$ or $\geq 3$ are called \emph{end points, ordinary points} and \emph{ramification points} and denoted by $E(X)$, $O(X)$ and $R(X)$, respectively. Also define \emph{vertices} as the points in $V(X):=E(X)\cup R(X)$. An \emph{edge} $J$ is any arc joining two points $p,q\in V(X)$ ($p=q$ is allowed) and containing no other vertices, we will write $J=pq$ and $(pq)=J-\{p,q\}$. 


Let $A,B\in C(X)$. An \textit{order arc} from $A$ to $B$ is a mapping (i.e. a continuous function) $\alpha :\left[0,1\right]\rightarrow C(X)$ such that $\alpha(0)=A$, $\alpha(1)=B$, and $\alpha(r)\subsetneq \alpha(s)$ whenever $r<s$. Existence and other properties of order arcs if $A\subsetneq B$ can be found in \cite[1.2-1.8]{Nadler(1978)}.

Given a positive integer $n$, we say that a space $X$ is \textit{$\frac{1}{n}$-homogeneous} provided that the natural action of the group of homeomorphisms of the space $X$ onto itself has exactly $n$ orbits. In this case we say that \textit{$n$ is the homogeneity degree of $X$}.

Given continua $X$ and $Y$, we write $X\approx Y$ if there exists an homeomorphism between $X$ and $Y$.

For a mapping $f:X\rightarrow Y$ between continua, we consider the \textit{induced mapping by $f$}, $C(f):C(X)\rightarrow C(Y)$,  given by $C(f)(A)=f(A)$. In case that $f$ is a homeomorphism, for each $p\in X$, $\left.C(f)\right|_{C(p,X)}:C(p,X)\rightarrow C(f(p),Y)$ is a homeomorphism (see \cite[Lemma 3.4, p. 262]{Pellicer(2003)}). The following result is easy to prove.

\begin{proposition}\label{equality}
If $X$ is an $\frac{1}{n}$-homogeneous continuum, then $K(X)$ has size $m\leq n$.
\end{proposition}

Concerning Proposition \ref{equality}, the equality $m=n$ does not hold in general, for example if $X$ is a $2$-cell, $X$ is $\frac{1}{2}$-homogeneous, but for each $p\in X$, $C(p,X)$ is a Hilbert cube (\cite[Theorem 4, p. 221]{Eberhart(1978)}). \\

The following is an example of a finite graph $X$, such that $K(X)$ has smaller size than the homogeneity degree of $X$. 

\begin{example}\label{example}
In the Euclidean plane consider $X=S\cup L\cup J$ where:
\begin{itemize}
\item $S=\{(x,y):(x+2)^{2}+y^{2}=1\}$,
\item $L=\{(x,0):x\in \left[-1,1\right]\}$,
\item $J=\{(1,y):y\in\left[-1,1\right]\}$.
\end{itemize}

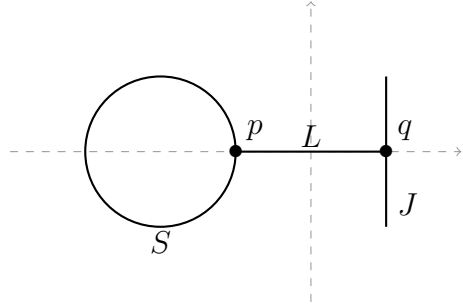
\[
\begin{tikzpicture}
\draw[thick] (-2,0) circle (1cm);  
\draw[->,dashed, gray!80] (-4,0) -- (2,0);
\draw[->,dashed, gray!80] (0,-2) -- (0,2);
\draw[thick] (-1,0) node {$\bullet$} node[above right]{$p$} -- (1,0) node{$\bullet$} node[above right]{$q$};
\draw node at (-2,-1.2){$S$};
\draw[thick] (1,-1) node[above right]{$J$} -- (1,1);
\draw node at (0,0.2) {$L$};
\end{tikzpicture}
\]

%
%
%
 
In this case, $K(X)$ has size at most 5 (using the results of the following section and Theorem \ref{pseudosymmetric} it can be verified that $K(X)$ has size exactly 5) and it is easy to see that $X$ is $\frac{1}{6}$-homogeneous.
\end{example}

\section{$C(p,X)$ of finite graphs}

In this section we suppose that all finite graphs $X$ have the metric, $d$, given by the arc length, i.e. if $x,y\in X$ the distance in $X$ from $x$ to $y$ will be the length of a shortest path connecting $x$ and $y$ in $X$. We will also suppose that each edge has length equal to $1$. In this case, if $x,y\in X$ belongs to the same edge, we say that $z$ is the midpoint of $xy$ if $d(x,z)=d(z,y)$. 

\begin{lemma}\label{finalesordinarios}
Let $X$ be a finite graph, $p\in X$ and $A\in C(p,X)$.
\begin{itemize}
\item[(i)] If $p\in E(X)$, then there exists an arc $L\subset C(p,X)$ such that $A\in L$.
\item[(ii)] If $p\in O(X)$, then there exists a $2$-cell $\mathcal{A}\subset C(p,X)$ such that $A\in \mathcal{A}$.
\end{itemize}
\end{lemma}

\begin{proof}
In order to prove $(i)$ take an order arc  $L$ in $C(p,X)$ from $\{p\}$ to $X$ through $A$ (see \cite[Theorem 14.6, p. 112]{Illanes(1999)}).

Let $p\in O(X)$. If $X$ is a simple closed curve it is easy to see that $C(p,X)$ is homeomorphic to an $2$-cell. So we  can suppose that $X$ is not a simple closed curve.  We consider then two cases:
\begin{enumerate}
\item If $A-\{p\}$ is not connected, let $C_1$, $C_2$ be the components of $A-\{p\}$. For $i\in\{1,2\}$, let $\alpha_i :[0,1]\to C(p,X) $ be an order arc from $\{p\}$ to $C_i\cup \{p\}$ and define $h:[0,1]\times [0,1]\to C(p,X)  $ as $h(s,t)=\alpha_1(s)\cup\alpha_2(t)$. Note that $h$ is a embedding. Thus $\mathcal{A}=h([0,1]\times [0,1])$ is a $2$-cell contained in $C(p,X)$ and $h(1,1)=A$.

\item If $A-\{p\}$ is connected, let $J=vw$ be the edge of $X$ containing $p$, where $v$ and $w$ are vertices of $X$. 

In case that $p$ is an end point of $A$, we can assume, without loss of generality, that there exists an arc $L\subset pw$ such that $L\cap A=\{p\}$. Let $\alpha,\beta:[0,1]\to C(p,X)$ be order arcs from $\{p\}$ to $A$, and from $\{p\}$ to $L$, respectively. Define $h:[0,1]\times [0,1]\to C(p,X)$ as $h(s,t)=\alpha(s)\cup \beta(t)$ and we conclude as above since $h(1,0)=A$.

If $p\in O(A)$,  let $m$ be the midpoint of $pw$ and let $a,b$ be the midpoints of $pm$ and $mw$, respectively. Since $A-\{p\}$ is connected, $A-ab$ is connected and contains $p$. Let $C$ be the closure in $X$ of $A-ab$. Let $\alpha:[0,1]\to C(p,X)$ be an order arc from $C$ to $C\cup am$ and $\beta:[0,1]\to C(p,X)$ and order arc from $C$ to $C\cup mb$. By defining $h:[0,1]\times\ [0,1]\to C(p,X)$ as $h(s,t)=\alpha(s)\cup \beta(t)$ we obtain the result since $h(1,1)=C\cup am\cup mb=A$.
\end{enumerate}
\end{proof}

Let $k\in \mathbb{N}$. A continuum $Y$ is a \textit{$k$-od} provided there exists $M\in C(Y)$ such that $Y-M$ has at least $k$ components. In this case, we will say that $M$ is a \textit{core} of the $k$-od. It is trivial that any $(k+1)$-od is also an $k$-od. The reader should take care that even in the case of finite graphs, for a point $p$ in $X$, there is subtile difference between $ord(p,X)=k$ and $p$ belongs to the core of a $k$-od in $X$, for example, each point in a simple closed curve has order equal $2$, but the simple closed curve is not a $2$-od. In the Example \ref{example}, let $M$ be
\[
M=\{(x,y)\in X : -2\leq x<1\}\cup \{q\}. 
\]
Here $M$ is the core of an $3$-od containing $(0,0)$ and $(0,0)$ has order equal $2$.\\

Related to these concepts we present the following well known result about $k$-cells in $C(p,X)$ (cf. \cite[Corollary 3.12, p. 1006]{Pellicer(2007)}).

\begin{theorem}
Given a continuum $X$ and $p\in X$, $C(p,X)$ contains a $k$-cell if and only if the point $p$ is contained in the core of a $k$-od.
\end{theorem}

For the class of finite graphs, the following result is similar to the last one and can be considered a generalization in some particular cases. 

\begin{theorem}\label{maintheorem}
Let $X$ be a finite graph and $p\in X$. Then for each $A\in C(p,X)$ and $\varepsilon >0$, there exists a subset $\mathcal A\subset C(p,X)$ such that $\mathcal A$ contains  an $n$-cell, where $n=\textrm{ord}(p,X)$ and $H_2(\mathcal A,\{A\})<\varepsilon$.
\end{theorem}
\begin{proof}
By Lemma \ref{finalesordinarios} it is only necessary to consider the case when $p\in R(X)$.

Suppose $0<\varepsilon <\min\{d(x,y):x,y\in V(X),x\neq y\}/4$. Let $L_1, L_2, \cdots , L_n$ be the edges of $X$ emanating from $p$. For each component $C$ of $A- \{p\}$, note that $C\cup \{p\}$ is a subgraph of $X$ such that either 
\begin{enumerate}[i)]
\item $p\in R(C\cup \{p\})\cup O(C\cup \{p\})$ or
\item $p\in E(C\cup \{p\})$.
\end{enumerate}
If $p\in O(C\cup\{p\})\cup R(C\cup \{p\})$, then $m(C):=ord(p,C\cup \{p\})\geq 2$. Without loss of generality, suppose that $L_1, L_2, \cdots , L_{m(C)}$ are edges of $X$ contained in $C\cup \{p\}$. Let $m_i$ be the midpoint of $L_i$. For  $i\in \{ 1,2,\cdots , m(C)-1\}$, let $a_i,b_i\in L_i$ such that $m_i\in a_ib_i$ and $d(a_i,m_i)=d(m_i,b_i)=\varepsilon /2$. Let $G_C=(C\cup\{p\})\backslash \bigcup \limits _{i=1}^{m(C)-1}(a_ib_i)$, this is clearly a continuum. Moreover, for each $D\in C(G_C,C\cup\{p\})$, $H(D,C\cup \{p\})<\varepsilon$.

For each $i\in \{1,2,\cdots , m(C)-1\}$, let $a'_i,b'_i$ be the midpoints of $a_im_i$ and $m_ib_i$, respectively. Let $J_i^a(t)$ be the arc in $a_ia'_i$ containing $a_i$, such that the length of $J_i^a(t)$ is equal $t$, for $t\in [0,\varepsilon /4]$. We define $J_i^b(t)$ in a similar way, for $b_ib'_i$. 
Let $\gamma _i: [0,\varepsilon /4]^2\longrightarrow C(G_C,C\cup\{p\})$ be defined as $\gamma _i(s,t)=G_C\cup J_i^a(s)\cup J_i^b(t)$.
 Define $\gamma_C :[0,\varepsilon /4]^{(m(C)-1)}\times [0,\varepsilon /4]^{(m(C)-1)}\longrightarrow C(G_C, C\cup\{p\})$ as
\[
\gamma _C (\overrightarrow{s},\overrightarrow{t}) =\bigcup \limits _{i=1}^{m(C)-1} \gamma _i(s_i, t_i),
\]
where $\overrightarrow{s}=(s_1,s_1,\cdots , s_{m(C)-1})$, $\overrightarrow{t}=(t_1,t_1,\cdots , t_{m(C)-1})$. 
It is easy to see that $\gamma _C$ is an embedding.
 Let $\mathcal{C}=\{C : \text{ is component of } A-\{p\} \text{ satisfying i) }  \}$, $l=|\mathcal C|$ and  $m=\sum \limits _{C\in \mathcal C} m(C)$; set $\mathcal C=\{C_1,\cdots , C_l\}$.
 
Let $\gamma: [0,\varepsilon /4]^{2m-l}\longrightarrow \bigcup \limits _{i=1}^l C(G_{C_i}, C_i\cup \{p\})$ given by
\[
\gamma (\overrightarrow{s_1}, \cdots ,\overrightarrow{s_l}, \overrightarrow{t_1}, \cdots ,\overrightarrow{t_{l}}) 
= \bigcup \limits _{i=1}^l \gamma _{C_i} (\overrightarrow{s_i}, \overrightarrow{t_i}) .
\] 
It is not difficult to prove that $\gamma$ is an embedding.
 
If $p\in E(C\cup\{p\})$, take an order arc $\alpha _C$ in $C(C\cup\{p\})$ from $\{p\}$ to $C\cup\{p\}$. By the continuity of $\alpha _C$ in $1$ there exists $\delta _C >0$ such that if $1-\delta _C \leq t\leq 1$, then $H(\alpha _C (t), C\cup\{p\})<\varepsilon$. \\
Let $\mathcal D =\{C: C \text{ is component of }A-\{p\} \text{ satisfaying ii) } \}$ and $k=|\mathcal D|$. \\
Let $\alpha:\displaystyle\prod_{C\in \mathcal D}[1-\delta_C,1]\to \displaystyle\bigcup_{C\in\mathcal D}C(C\cup \{p\})$ defined as $\alpha(\overrightarrow{x})=\displaystyle\bigcup_{C\in\mathcal D}\alpha(x_C)$ for $\overrightarrow{x}=(x_C)_{C\in \mathcal D}\in \displaystyle\prod_{C\in \mathcal D}[1-\delta_{C},1]$. It is not difficult to prove that $\alpha$ is an embedding.


Let $L_{n_1}, L_{n_2}, \cdots , L_{n_r}$ be the edges such that there exist an arc $Z_i\subset L_{n_i}$ with $Z_i\cap A=\{p\}$. Observe that $r=n-(m+k)$.  For each $i\in \{1,2,\cdots , r\}$ let $\beta_i$ be an order arc in $C(Z_i)$ from $\{p\}$ to $Z_i$.

 
 By continuity of $\beta _i$ in $0$, there exists $\delta _i>0$ such that if $0\leq t\leq \delta _i$, then $H(\{p\}, \beta _i(t))<\varepsilon$ and $\beta _i(t)\cap A=\{p\}$. Let $\beta:\displaystyle\prod^{r}_{i=1}[0,\delta_{i}]\to \displaystyle\bigcup^{r}_{i=1}L_{n_{i}}$ de defined as $\beta(\overrightarrow{x})=\displaystyle\bigcup^r_{i=1}\beta_i (x_i)$ for $\overrightarrow{x}=(x_1, \cdots,x_r)\in \displaystyle\prod_{i=1}^r[0,\delta_i]$. It is easy to prove that $\beta$ is an embedding.

Let  $h_{2m-l}:[0,1]^{2m-l}\longrightarrow [0,\varepsilon/4]^{2m-l}$, $h_k:[0,1]^k\longrightarrow \prod \limits _{C\in \mathcal D}[1-\delta _C , 1]$ and $h_k:[0,1]^r\longrightarrow \prod \limits _{i=1}^r[0,\delta _i]$ be homeomorphisms. 

We define $h:[0,1]^{2m-l}\times [0,1]^k\times [0,1]^r\longrightarrow C(p,X)$ as 
\[
 h(\overrightarrow{x}, \overrightarrow{y}, \overrightarrow{z})= \gamma \circ h_{2m-l}(\overrightarrow{x}) \cup \alpha \circ h_k(\overrightarrow{y}) \cup \beta \circ h_r(\overrightarrow{z})
\]
where $\overrightarrow{x}\in [0,1]^{2m-l}$, $\overrightarrow{y}\in [0,1]^{k}$ and $\overrightarrow{z}\in [0,1]^{r}$. Note that $h$ is an embedding, thus, $\mathcal A: =Im (h)$, the image of $h$, is a $(2m-l+k+r)$-cell contained in $C(p,X)$. Moreover, by the metric in $X$, $H_2(\mathcal A, \{A\})<\varepsilon$. Since $m_i(C) \geq 2$, and $m=\sum \limits _{i=1}^l m_i(C) \geq 2l$,
\begin{align*}
2m-l+k+r & = 2m -l +k + (n-k-m) \\
 & = m+n-l \geq n .
\end{align*}
Then $\mathcal A$ contains an $n$-cell.
\end{proof}

\begin{corollary}
Let $X$ be a finite graph and $p,q\in X$ such that $C(p,X)\approx C(q,X)$. Then $\textrm{ord}(p,X)=\textrm{ord}(q,X)$.
\end{corollary}

\begin{proof}
Suppose that $n={ord}(p,X)< ord (q,X)=m$ and let $h:C(p,X)\rightarrow C(q,X)$ be a homeomorphism. Since $\{p\}$ has a neighborhood base in $C(p,X)$ consisting of $n-\textrm{cells}$, $h(\{p\})$ has a neighborhood base in $C(q,X)$ consisting of $n-\textrm{cells}$ which is impossible by Theorem \ref{maintheorem}, thus $m\leq n$. A similar argument implies that $n\leq m$. Thus, $n=m$.
\end{proof}

The following is an  easy consequence of the previous corollary, although it is trivial, illustrates necessary conditions for $C(p,X)$ and $C(q,X)$ to be homeomorphic.
 
\begin{remark}\label{orders}
Let $X$ be a finite graph and $p,q\in X$ such that $C(p,X)\approx C(q,X)$.  Let $\mathcal H (X)$ be one of $E(X), O(X)$ or $R(X)$. It holds that if $p\in \mathcal H(X)$ then also $q\in \mathcal H(X)$.
\end{remark}

Concerning previous remark, observe that if $X$ is a finite graph such that $K(X)$ has size $1$, then each point in $X$ is an ordinary point, thus, by \citep[Proposition 9.5]{Nadler(1992)} we conclude that $X$ is a simple closed curve. We rewrite this in the next result, which gives a characterization of simple closed curves.

\begin{corollary}\label{simpleclosedcurve}
For a finite graph $X$ the following conditions are equivalent:
\begin{enumerate}
\item $X$ is a simple closed curve,
\item $K(X)$ has size $1$,
\item $X$ is homogeneous.
\end{enumerate}
\end{corollary}

Let $X$ be a finite graph. If $q\in E(X)$ and $X$ is not an arc, we denote by $v(q)$  the unique point in $R(X)$ such that the component $C$, of $X-\{v(q)\}$, containing $q$, satisfies that $C\cup\{v(q)\}$ is an arc.
\begin{proposition}\label{neighbor}
Let $X$ be a finite graph which is not an arc. If $e_1,e_2\in E(X)$ and $C(e_1,X)\approx C(e_2,X)$ then $\textrm{ord}(v(e_1),X)=\textrm{ord}(v(e_2),X)$.
\end{proposition}

\begin{proof}
Let $h:C(e_{1},X)\rightarrow C(e_{2},X)$ be a homeomorphism. Denote by $l_{1}$ and $l_{2}$, the edges of $X$, $e_{1}v(e_{1})$ and $e_{2}v(e_{2})$, respectively. By Theorem \ref{maintheorem} we have that for each $A\in C(e_{1},l_{1})-\{l_{1}\}$, $h(A)\in C(e_{2},l_{2})$. By the continuity of $h$, $h(l_{1})\in C(e_{2},l_{2})$. Again by Theorem \ref{maintheorem}, $h(l_{1})=l_{2}$ and ${ord}(v(e_{1}),X)={ord}(v(e_{2}),X)$.
\end{proof}

\begin{definition}

Let $X$ be an finite graph, $p\in O(X)$ and $v,w\in V(X)$ are the end points of the edge containing $p$. Define
\[
\Sigma (p,X) =\begin{cases}
{ord}(v,X) + {ord}(w,X), & \text{ if   } v\neq w , \\
{ord}(v,X), & \text{ if  } v=w.
\end{cases}
\]
\end{definition}

The following lemma is a particular case of \cite[Corollary 3.5, p. 1006]{Pellicer(2007)}.

\begin{lemma}\label{neigboredges}
Let $X$ be a finite graph. If $A$ is an edge of $X$ and $p\in O(X)\cap A$, then there exist $\mathcal{A}\subset C(p,X)$ such that $\mathcal{A}$ is contained in a $\Sigma(p,X)$-cell.
\end{lemma}

\begin{proposition}\label{neighbors}
Let $X$ be a finite graph. If $p,q\in O(X)$ and $C(p,X)\approx C(q,X)$ then $\Sigma (p,X)=\Sigma (q,X)$.
\end{proposition}

\begin{proof}
Let $h:C(p,X)\rightarrow C(q,X)$ be a homeomorphism and let $l_{1}$ and $l_{2}$ be the edges containing $p$ and $q$, respectively. Suppose that $a,b\in l_{1}$ and $c,d\in l_{2}$ are vertices in $X$. Using Theorem \ref{maintheorem} we can see that for each $A\in C(p,l_{1})$ such that $A\subset l_{1}-\{a,b\}$, $h(A)\in C(q,l_{2})$. Then, by the continuity of $h$, $h(l_{1})\subset l_{2}$. By Theorem \ref{maintheorem}, we have that $h(l_{1})=l_{2}$ and by Lemma \ref{neigboredges} this implies that $\Sigma(p,X)=\Sigma(q,X)$.
\end{proof}


\section{Finite graphs with $K(X)$ having size $n$} \label{graficas}

The purpose of this section is to use the previous results to construct for each positive integer $n$, a finite graph $X_{n}$ such that $K(X_{n})$ has size $n$. For this, we will start introducing notation for some sets in the Euclidean Plane, $\mathbb{R}^2$. For each $n\in\mathbb{N}$:
\begin{itemize}
\item $I_{n}=\left[n-1,n\right]\times \{0\}$;
\item $C_{n}$ will denote the circle with radius $\frac{1}{4}$ and center $(n,\frac{1}{4})$;
\item $C$ will denote the circle with radius $\frac{1}{8}$ and center $(2,\frac{1}{8})$;
\item for each $i\in\mathbb{N}$, we consider $l^{(n)}_{i}$ as the linear segment joining $(n,0)$ and $\left(n-\frac{1}{i},-1\right)$;
\end{itemize}

\begin{minipage}[l]{0.5\textwidth}
 \begin{itemize}
\item if $n\geq 3$, $S_{n}=I_{n}\cup C_{n}\cup \displaystyle\bigcup^{2n}_{i=1}l^{(n)}_{i}$;
\item $P_{1}=I_{2}\cup C_{2}\cup C$ ;
\item $P_{2}=I_{2} \cup C_{1}\cup C_{2}\cup C$;
 \end{itemize}
\end{minipage}
\begin{minipage}[r]{0.5\textwidth}
\begin{itemize}
\item $P_{3}=I_{2}\cup C_{2}\cup C \cup \displaystyle\bigcup^{3}_{i=1}l^{(1)}_{i}$;
\item $P_{4}=I_{2}\cup C_{1}\cup C_{2}\cup C\cup l^{(1)}_{1}$;
\item $P_{5}=I_{2}\cup C_{2}\cup C\cup l^{(1)}_{1}\cup l^{(1)}_{2}\cup l^{(2)}_{1}\cup l^{(2)}_{2}$.
\end{itemize}
\end{minipage}

Next we give a picture of these sets.
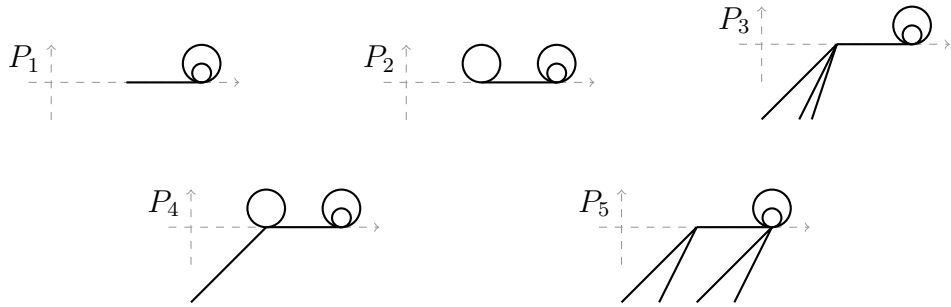
\[
\begin{tikzpicture}
\draw[->,dashed, gray!80] (-0.3,0) -- (2.5,0);
\draw[->,dashed, gray!80] (0,-0.5) -- (0,0.5);
\draw (0,0) node[above left]{$P_1$};
\begin{scope}[thick]
\draw (1,0) -- (2,0);
\draw (2,0.25) circle (0.25cm);
\draw (2,1/8) circle (0.125cm);
\end{scope}
\end{tikzpicture}
\hspace{1.5cm}
\begin{tikzpicture}
\draw[->,dashed, gray!80] (-0.3,0) -- (2.5,0);
\draw[->,dashed, gray!80] (0,-0.5) -- (0,0.5);
\draw (0,0) node[above left]{$P_2$};
\begin{scope}[thick]
\draw (1,0) -- (2,0);
\draw (1,1/4) circle (0.25cm) (2,1/4) circle (0.25cm);
\draw (2,1/8) circle (0.125cm);
\end{scope}
\end{tikzpicture}
\hspace{1.5cm}
\begin{tikzpicture}
\draw[->,dashed, gray!80] (-0.3,0) -- (2.5,0);
\draw[->,dashed, gray!80] (0,-0.5) -- (0,0.5);
\draw (0,0) node[above left]{$P_3$};
\begin{scope}[thick]
\draw (1,0) -- (2,0);
\draw (2,1/4) circle (0.25cm);
\draw (2,1/8) circle (0.125cm);
\foreach \x in {0,1-1/2,1-1/3} {
\draw (1,0)-- (\x,-1);
}
\end{scope}
\end{tikzpicture}
\]

\[
\begin{tikzpicture}
\draw[->,dashed, gray!80] (-0.3,0) -- (2.5,0);
\draw[->,dashed, gray!80] (0,-0.5) -- (0,0.5);
\draw (0,0) node[above left]{$P_4$};
\begin{scope}[thick]
\draw (1,0) -- (2,0);
\draw (1,1/4) circle (0.25cm);
\draw (2,1/4) circle (0.25cm);
\draw (1,0)-- (0,-1);
\draw (2,1/8) circle (0.125cm);
\end{scope}
\end{tikzpicture}
\hspace{2.5cm}
\begin{tikzpicture}
\draw[->,dashed, gray!80] (-0.3,0) -- (2.5,0);
\draw[->,dashed, gray!80] (0,-0.5) -- (0,0.5);
\draw (0,0) node[above left]{$P_5$};
\begin{scope}[thick]
\draw (1,0) -- (2,0);
\draw (2,1/4) circle (0.25cm);
\draw (1,0)-- (0,-1) (1,0) -- (1/2,-1);
\draw (2,0)-- (1,-1) (2,0) -- (2-1/2,-1);
\draw (2,1/8) circle (0.125cm);
\end{scope}
\end{tikzpicture}
\]
We define $X_{n}$ as follows:
\begin{enumerate}
\item Let $X_{1}$ be a simple closed curve, $X_{2}$ an arc and $X_{3}$ a simple triod. 
\item $X_{i+3}=P_i$ for $i\in \{1,2,3,4,5\}$, 
\item for each $n\geq 9$ we define $X_n$ recursively as follows, write $n$ as $n=5(k+1)+r$, where $r\in \{-1,0,1,2,3\}$, then 
$$X_{n}=X_{5k+r}\cup S_{k+2}.$$
\end{enumerate}

To clarify the last formula in 3 we draw some pictures.
\bigskip
\[
\begin{tikzpicture}[scale=0.8]
\draw[->,dashed, gray!80] (-0.5,0) -- (3,0);
\draw[->,dashed, gray!80] (0,-1) -- (0,1);
\draw node at (-0.5,0.5){$X_4$};
\begin{scope}[very thick]
\draw (1,0) -- (2,0);
\draw (2,0.25) circle (0.25cm);
\draw (2,1/8) circle (0.125cm);
\end{scope}
\end{tikzpicture}
\qquad
\begin{tikzpicture}[scale=0.8]
\draw[->,dashed, gray!80] (-0.5,0) -- (3.5,0);
\draw[->,dashed, gray!80] (0,-1) -- (0,1);
\draw node at (-0.5,0.5){$X_9$};
\begin{scope}[thick]
\draw (1,0) -- (2,0);
\draw (2,0.25) circle (0.25cm);
\draw (2,0) -- (3,0);
\draw (3,0.25) circle (0.25cm);
\foreach \x in {2,2.5,8/3, 3-1/4,3-1/5,3-1/6} {
\draw (3,0) -- (\x,-1);
\draw (2,1/8) circle (0.125cm);
}
\end{scope}
\end{tikzpicture}
\qquad
\begin{tikzpicture}[scale=0.8]
\draw[->,dashed, gray!80] (-0.5,0) -- (4.5,0);
\draw[->,dashed, gray!80] (0,-1) -- (0,1);
\draw node at (-0.5,0.5){$X_{14}$};
\begin{scope}[thick]
\draw (1,0) -- (2,0);
\draw (2,0.25) circle (0.25cm);
\draw (2,0) -- (3,0);
\draw (3,0.25) circle (0.25cm);
\foreach \x in {2,2.5,8/3, 3-1/4,3-1/5,3-1/6} {
\draw (3,0) -- (\x,-1);
\draw (2,1/8) circle (0.125cm);
}
\draw (3,0) -- (4,0);
\draw (4,0.25) circle (0.25cm);
\foreach \x in {3,3.5,4-1/3, 4-1/4,4-1/5,4-1/6,4-1/7} {
\draw (4,0) -- (\x,-1);
}
\end{scope}
\end{tikzpicture}
\]
\[
\begin{tikzpicture}[scale=0.8]
\draw[->,dashed, gray!80] (-0.5,0) -- (3,0);
\draw[->,dashed, gray!80] (0,-1) -- (0,1);
\draw node at (-0.5,0.5){$X_5$};
\begin{scope}[thick]
\draw (1,0) -- (2,0);
\draw (1,1/4) circle (0.25cm) (2,1/4) circle (0.25cm);
\draw (2,1/8) circle (0.125cm);
\end{scope}
\end{tikzpicture}
\qquad
\begin{tikzpicture}[scale=0.8]
\draw[->,dashed, gray!80] (-0.5,0) -- (3.5,0);
\draw[->,dashed, gray!80] (0,-1) -- (0,1);
\draw node at (-0.5,0.5){$X_{10}$};
\begin{scope}[thick]
\draw (1,0) -- (2,0);
\draw (1,1/4) circle (0.25cm) (2,1/4) circle (0.25cm);
\draw (2,1/8) circle (0.125cm);
\draw (2,0.25) circle (0.25cm);
\draw (2,0) -- (3,0);
\draw (3,0.25) circle (0.25cm);
\foreach \x in {2,2.5,8/3, 3-1/4,3-1/5,3-1/6} {
\draw (3,0) -- (\x,-1);
}
\end{scope}
\end{tikzpicture}
\qquad
\begin{tikzpicture}[scale=0.8]
\draw[->,dashed, gray!80] (-0.5,0) -- (4.5,0);
\draw[->,dashed, gray!80] (0,-1) -- (0,1);
\draw node at (-0.5,0.5){$X_{15}$};
\begin{scope}[thick]
\draw (1,0) -- (2,0);
\draw (1,1/4) circle (0.25cm) (2,1/4) circle (0.25cm);
\draw (2,1/8) circle (0.125cm);
\draw (2,0.25) circle (0.25cm);
\draw (2,0) -- (3,0);
\draw (3,0.25) circle (0.25cm);
\foreach \x in {2,2.5,8/3, 3-1/4,3-1/5,3-1/6} {
\draw (3,0) -- (\x,-1);
}
\draw (3,0) -- (4,0);
\draw (4,0.25) circle (0.25cm);
\foreach \x in {3,3.5,4-1/3, 4-1/4,4-1/5,4-1/6,4-1/7} {
\draw (4,0) -- (\x,-1);
}
\end{scope}
\end{tikzpicture}
\]

\bigskip
By using Remark \ref{orders}, and Propositions \ref{neighbor} and \ref{neighbors}, we obtain the following result. 

\begin{theorem}\label{graphs}
For each $n$, $K(X_{n})$ has size $n$. 
\end{theorem}

Concerning Theorem \ref{graphs}, note that for each $n$, $X_{n}$ is a $\frac{1}{n}$-ho\-mo\-ge\-ne\-ous continuum. On the other hand Example \ref{example} shows a $\frac{1}{6}$-homogeneous finite graph $X$, such that $K(X)$ has size $5$.  We will prove this fact in the next section by proving a more general result. 

By using the same ideas of the previous section, it can be proved that the continuum in the next picture has countably many topologically distinct types of $C(p,X)$.
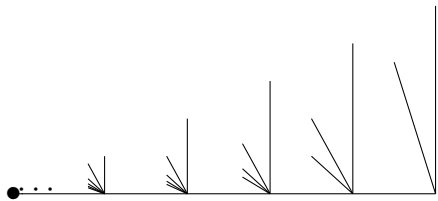
\[
\begin{tikzpicture}[xscale=1.1, yscale=0.5]
\foreach \x in {0,1,2,3,4,5}{
\draw (\x,0) -- (\x,\x);
}
\draw (-0.1,0) node{$\bullet$} -- (5,0);
\draw node at (0.2,0.1){$\cdots$};
\foreach \x in {2,...,5}{
\foreach \y in {1,2,...,\x}{
\draw (6-\x,0) -- ({(6-\x)-1/(\x)},{(1/\x)*(4/\y)});
}
}
\draw (5,0)  -- (4.5, 3.5);
\end{tikzpicture}
\]

\begin{question}
Does there exist a continuum $X$ such that the cardinality of $K(X)/\sim $ is equal to $\mathfrak c$, the cardinality of the continuum? 
\end{question}

By the arguments given above the Proposition \ref{equality},  such a continuum must have uncountably many orbits.

It is known by \cite[Theorem 9]{KennedyRogers(1986)} that the pseudocircle  $X$ has uncountably many orbits, however since it is hereditarily indecomposable each element of $K(X)$ is an arc (\cite[Lemma 3.19]{Pellicer(2003)}).
%

\section{$\frac{1}{n}$-homogeneous continua with size smaller than $n$}

We have seen in the last section that for every positive integer $n$ there exist a finite graph  $X_n$ with $K(X_n)$ having size $n$. In our construction, all the examples $X_n$ are $1/n$-homogeneous. Since the homogeneity degree is an upper bound for the size of $K(X)$, it is natural to ask if it is possible to give families of $1/n$- homogeneous finites graphs $Y_n$ such that $K(Y_n)$ has size $<n$.

\begin{definition}
Given a continuum $X$ and $p,q\in X$ we say that $X$ is \textbf{pseudo--symmetric with respect to $p$ and $q$}, provided that there exists a homeomorphism $\varphi : C(p,X)\to C(q,X)$ such that $\varphi (\{p\})=\{q\}$ and $\varphi(X)=X$.
\end{definition}

As examples of pseudo--symmetric continua we have the following:
\begin{itemize}
\item The continuum $X$ in Example \ref{example} is pseudo--symmetric with respect to the end points $(1,1)$ and $(1,-1)$. 
\item Each homogeneous continuum is pseudo--symmetric with respect to each pair of its points. 
\item By \cite[Lemma 3.19]{Pellicer(2003)} every hereditarily indecomposable continuum is pseudo--symmetric with respect to each pair of its points. 

\end{itemize}

\begin{theorem}\label{pseudosymmetric}
Let $Y$ be a pseudo--symmetric continuum with respect to $p$ and $q$. Let $X=L\cup Y\cup K$, where $L$ and $K$ are continua such that $L\cap Y=\{p\}$, $K\cap Y=\{q\}$ and $K\cap L=\emptyset$. Suppose that there exists a homeomorphism $f:C(p,L)\to C(p,K)$ such that $f(\{p\})=\{q\}$ and $f(L)=K$. Then $C(p,X)$ is homeomorphic to $C(q,X)$. 
\end{theorem} 

\begin{proof}
Since $K$ and $L$ are disjoint continua, $L\cap Y=\{p\}$ and $K\cap Y=\{q\}$, then $p$ and $q$ are cut points of $X$. Moreover $p$ is a cut point of $L\cup Y$ and $q$ is a cut point of $K\cup Y$. Thus it is easy to see that the function $g: C(Y,X)\to C(p,L)\times C(q,K)$ defined as $g(A)= (A\cap L, A\cap K)$, for each $A\in C(Y,X)$ is a homeomorphism.

Since $Y$ is pseudo--symmetric with respect to $p$ and $q$, there exists a homeomorphism $\varphi : C(p,Y)\to C(q,Y)$ such that $\varphi (\{p\})=\{q\}$ and $C(Y)=Y$.

Let $h:C(p,X)\rightarrow C(q,X)$ be defined as 
\[
h(A) =\begin{cases}
\varphi(A\cap Y)\cup f(A\cap L), & \text{if  } A\in C(p,L\cup Y); \\
g^{-1}(f^{-1}(A\cap K), f(A\cap L)), & \text{if  } A\in C(Y,X).
\end{cases}
\]
Note that if $A\in C(p,L\cup Y)\cap C(Y,X)$ then $A\cap K=\{q\}$, and $A\cap Y=Y$. Thus
\begin{align*}
g^{-1}(f^{-1}(A\cap K), f(A\cap L)) & =g^{-1}(f^{-1}(\{q\}), f(A\cap L))\\
&=Y\cup f(A\cap L)=\varphi(A\cap Y)\cup f(A\cap L).
\end{align*}
Hence, $h$ is well-defined and since $\varphi$, $f$ and $g$ are homeomorphisms, we conclude that $h$ is a homeomorphism.

%
%
%
\end{proof}

Using the notation of the previous theorem, if we suppose that in addition $X$ is $\frac{1}{n}$-homogeneous and, $p$ and $q$ are in different orbits under the accion of the homeomorphism group, then $K(X)$ has size smaller than $n$. As a consequence, it is easy to see that the continuum $X$ in Example \ref{example}, is pseudo--symmetric with respect to $p$ and $q$.

Next we will construct a familly of finite graphs with hyperspace $K(X)$ having size $n$ and homogeneity degree greater than $n$. The main idea is to use the Theorem \ref{pseudosymmetric} pasting
 two disjoint continua $L$ and $K$ in a pseudo--symmetric continuum $X$ in points $p$ and $q$, respectively, such that $C(p,L)$ and $C(q,K)$ are homeomorphic, but in such a way that there is no homeomorphism between $L$ and $K$ sending $p$ to $q$. In order to do this, observe that each one of the graphs $P_{i}$ given in the last section are pseudo--symmetric with respect to the points $(2,\frac{1}{2})$ and $(2,\frac{1}{4})$, $i\in\{1,2,3,4,5\}$. We will make a modification to these graphs, attaching an arc and a circunference in the points $(2,\frac{1}{4})$ and $(2,\frac{1}{2})$, respectively, getting  the new graphs $Q_{i}$, wich are illustrated below:

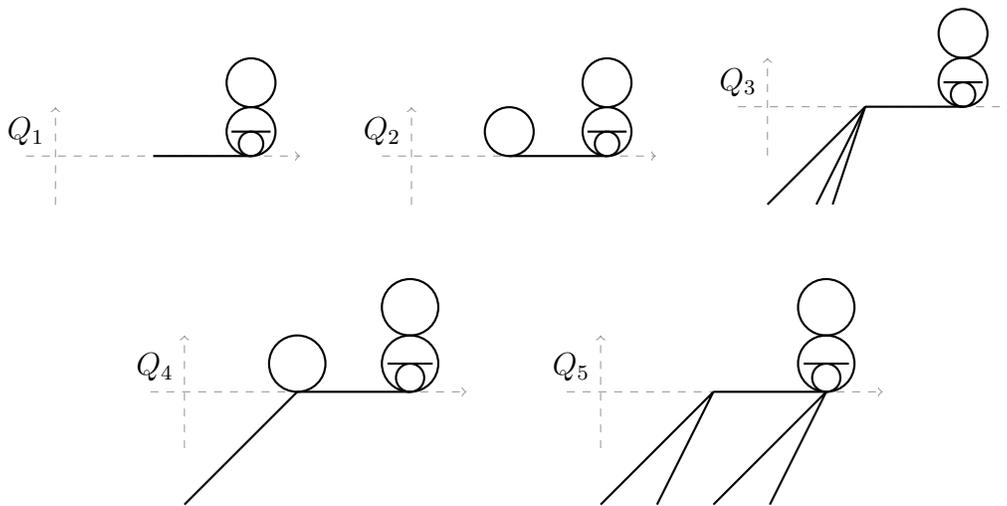
\[
\begin{tikzpicture}[scale=1.3]
\draw[->,dashed, gray!80] (-0.3,0) -- (2.5,0);
\draw[->,dashed, gray!80] (0,-0.5) -- (0,0.5);
\draw (0,0) node[above left]{$Q_1$};
\begin{scope}[thick]
\draw (1,0) -- (2,0);
\draw (1.8,0.25) -- (2.2,0.25);
\draw (2,0.25) circle (0.25cm);
\draw (2,0.75) circle (0.25cm);
\draw (2,1/8) circle (0.125cm);
\end{scope}
\end{tikzpicture}
\hspace{0.7cm}
\begin{tikzpicture}[scale=1.3]
\draw[->,dashed, gray!80] (-0.3,0) -- (2.5,0);
\draw[->,dashed, gray!80] (0,-0.5) -- (0,0.5);
\draw (0,0) node[above left]{$Q_2$};
\begin{scope}[thick]
\draw (1,0) -- (2,0);
\draw (1.8,0.25) -- (2.2,0.25);
\draw (1,1/4) circle (0.25cm) (2,1/4) circle (0.25cm);
\draw (2,1/8) circle (0.125cm);
\draw (2,0.75) circle (0.25cm);
\end{scope}
\end{tikzpicture}
\hspace{0.7cm}
\begin{tikzpicture}[scale=1.3]
\draw[->,dashed, gray!80] (-0.3,0) -- (2.5,0);
\draw[->,dashed, gray!80] (0,-0.5) -- (0,0.5);
\draw (0,0) node[above left]{$Q_3$};
\begin{scope}[thick]
\draw (1,0) -- (2,0);
\draw (1.8,0.25) -- (2.2,0.25);
\draw (2,1/4) circle (0.25cm);
\draw (2,1/8) circle (0.125cm);
\draw (2,0.75) circle (0.25cm);
\foreach \x in {0,1-1/2,1-1/3} {
\draw (1,0)-- (\x,-1);
}
\end{scope}
\end{tikzpicture}
%
\]

\[
%
\begin{tikzpicture}[scale=1.5]
\draw[->,dashed, gray!80] (-0.3,0) -- (2.5,0);
\draw[->,dashed, gray!80] (0,-0.5) -- (0,0.5);
\draw (0,0) node[above left]{$Q_4$};
\begin{scope}[thick]
\draw (1,0) -- (2,0);
\draw (1.8,0.25) -- (2.2,0.25);
\draw (1,1/4) circle (0.25cm);
\draw (2,1/4) circle (0.25cm);
\draw (2,0.75) circle (0.25cm);
\draw (1,0)-- (0,-1);
\draw (2,1/8) circle (0.125cm);
\end{scope}
\end{tikzpicture}
\hspace{1cm}
\begin{tikzpicture}[scale=1.5]
\draw[->,dashed, gray!80] (-0.3,0) -- (2.5,0);
\draw[->,dashed, gray!80] (0,-0.5) -- (0,0.5);
\draw (0,0) node[above left]{$Q_5$};
\begin{scope}[thick]
\draw (1,0) -- (2,0);
\draw (1.8,0.25) -- (2.2,0.25);
\draw (2,1/4) circle (0.25cm);
\draw (2,0.75) circle (0.25cm);
\draw (1,0)-- (0,-1) (1,0) -- (1/2,-1);
\draw (2,0)-- (1,-1) (2,0) -- (2-1/2,-1);
\draw (2,1/8) circle (0.125cm);
\end{scope}
\end{tikzpicture}
\]

Similar to the construction of the graphs $X_{n}$,  by using the same symbols as in the last section, we define for each $n\geq 4$, a finite graph $Y_{n}$ as follows:
\begin{enumerate}
\item $Y_{i+3}=Q_i$ for $i\in \{1,2,3,4,5\}$, 
\item for each $n\geq 9$, we write $n$ as $n=5(k+1)+r$, where $r\in \{-1,0,1,2,3\}$, and define 
$$Y_{n}=Y_{5k+r}\cup S_{k+2}.$$
\end{enumerate}
We consider furthermore the following finite graphs:

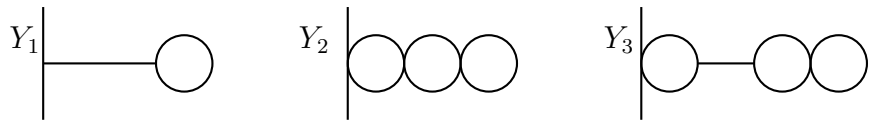
\[
\begin{tikzpicture}[scale=1.5]
\draw (1,0) node[above left]{$Y_1$};
\begin{scope}[thick, xshift=-2]
\draw (1,0) -- (2,0);
\draw (2.25,0) circle (0.25cm);
\draw (1,-.5) -- (1,.5);
\end{scope}
\end{tikzpicture}
\hspace{1cm}
\begin{tikzpicture}[scale=1.5]
\draw (0.9,0) node[above left]{$Y_2$};
\begin{scope}[thick, xshift=-1]
\draw (2.25,0) circle (0.25cm);
\draw (1,-.5) -- (1,.5);
\draw (1.25,0) circle (0.25cm);
\draw (1.75,0) circle (0.25cm);
\end{scope}
\end{tikzpicture}
\hspace{1cm}
\begin{tikzpicture}[scale=1.5]
\draw (0.5,0) node[above left]{$Y_3$};
\begin{scope}[thick,xshift=-1]
\draw (1,0) -- (1.5,0);
\draw (.5,-.5) -- (.5, .5);
\draw (2.25,0) circle (0.25cm);
\draw (.75,0) circle (0.25cm);
\draw (1.75,0) circle (0.25cm);
\end{scope}
\end{tikzpicture}
\]

By using Theorem \ref{pseudosymmetric}, we have that for each $n$, $Y_{n}$ is a $\frac{1}{n+5}$-homogeneous finite graph such that  $K(Y_{n})$ has size $n+4$. By this and Corollary \ref{simpleclosedcurve}, it remains to solve the following question.

\begin{question}
If $n\in\{2,3,4\}$, does there exists a $\frac{1}{n+1}$-homogeneous finite graph, $X$, such that $K(X)$ has size $n$?
\end{question}

%
%

We conclude this paper with the following natural problem, which is still open.
\begin{problem}
Characterize the class of finite graphs $X$, such that $X$ is $\frac{1}{n}$-homogeneous and $K(X)$ has size $n$, where $n$ is a positive integer.
\end{problem}

\textbf{Acknowledgments.}  All authors wish to thank the support given by FCFM-PFCE-2017-SEP and the third author thanks SEP for the support given by \emph{Apoyo a la Incorporaci\'on de Nuevos PTC} SEP/23-005/196387.\\
The authors are thankful to the referee for his/her useful suggestions, that improve the paper.


\end{document}